\documentclass[a4paper,11pt]{article}
\usepackage{latexsym}
\usepackage{amssymb}
\usepackage{enumerate}
\usepackage[shortlabels]{enumitem}
\usepackage{xcolor}
\usepackage{commath}
\usepackage{comment}
\usepackage{amsfonts}
\usepackage{amsmath,amsthm}
\usepackage{hyperref}
\usepackage{algorithm}
\usepackage{algorithmic}
\usepackage{relsize}
\usepackage{framed}
\usepackage{boites}
\usepackage[pdftex]{graphicx}
\newtheorem{mainthm}{Theorem}

\newtheorem{maincorol}{Corollary}

\newenvironment{@abssec}[1]{%
    \if@twocolumn

      \section*{#1}%
    \else

      \vspace{.05in}\footnotesize
      \parindent .2in
 {\upshape\bfseries #1. }\ignorespaces
    \fi}

    {\if@twocolumn\else\par\vspace{.1in}\fi}

\newenvironment{keywords}{\begin{@abssec}{\keywordsname}}{\end{@abssec}}

\newenvironment{AMS}{\begin{@abssec}{\AMSname}}{\end{@abssec}}

\newcommand\keywordsname{Key words}
\newcommand\AMSname{AMS subject classifications}
\newcommand\AMname{AMS subject classification}
\newcommand\restr[2]{{
\left.\kern-\nulldelimiterspace 
#1 
\vphantom{|} 
\right|_{#2} 
}}
\newtheorem{theorem}{Theorem}[section]
\newtheorem{lemma}[theorem]{Lemma}
\newtheorem{corollary}[theorem]{Corollary}

\newtheorem{remark}[theorem]{Remark}

\newcommand{\RR}{\mathbb{R}}

\def\XXint#1#2#3{{\setbox0=\hbox{$#1{#2#3}{\int}$}
\vcenter{\hbox{$#2#3$}}\kern-.5\wd0}}

\newcommand{\osc}{\mathop{\mathrm{osc}}}

\newcommand{\dist}{\mathop{\mathrm{dist}}}  
\newcommand{\link}{\mathop{\circ\kern-.35em -}}

\newcommand{\ol}{\overline}
\newcommand{\pa}{\partial}

\newcommand{\dv}{\mathop{\mathrm{div}}}
\newcommand{\na}{\nabla}
\newcommand{\nr}{\Vert}
\newcommand{\gr}{\nabla}

\newcommand{\al}{\alpha}
\newcommand{\be}{\beta}
\newcommand{\ga}{\gamma}  
\newcommand{\Ga}{\Gamma}
\newcommand{\de}{\delta}
\newcommand{\De}{\Delta}
\newcommand{\ve}{\varepsilon}
 
\newcommand{\la}{\lambda}
\newcommand{\La}{\Lambda}    

\newcommand{\si}{\sigma}

\newcommand{\om}{\omega}
\newcommand{\Om}{\Omega}
\newcommand{\rn}{{\mathbb{R}}^N}
\newcommand{\sg}{\sigma}

\newcommand\setbld[2]{\left\{ #1 \;\middle |\; #2\right\}}

\newcommand{\cD}{\mathcal{D}}

\title{\bf Quantitative stability estimates for a two-phase Serrin-type overdetermined problem
}

\author{Lorenzo Cavallina \, Giorgio Poggesi \, Toshiaki Yachimura
}

\date{}

\begin{document}

\maketitle

\begin{abstract}
In this paper, we deal with an overdetermined problem of Serrin-type
with respect to a two-phase elliptic operator in divergence form with
piecewise constant coefficients. In particular, we consider the case where the two-phase overdetermined problem is close to the one-phase setting. First, we show quantitative stability estimates for the two-phase
problem via a one-phase stability result. Furthermore, we prove non-existence for the corresponding inner problem by the aforementioned two-phase stability result.
\end{abstract}



\pagestyle{plain}
\thispagestyle{plain}

\section{Introduction and main results}\label{sec:introduction}
Let $\Om$ be a bounded domain of $\RR^N$ ($N \ge 2$) and let $D$ be an open set such that $\ol{D} \subset \Om$. In this paper, we consider the following two-phase Dirichlet boundary value problem: 
\begin{equation}\label{eq:2phase problem}
\begin{cases}
-\dv{(\si \na u)}= 1 \, \mbox{ in } \, \Om ,\\
u=0 \, \mbox{ on } \, \pa \Om, 
\end{cases}
\end{equation}
where $\si = \si(x)$ is the piecewise constant function defined by $\si(x) = 1 + (\si_c - 1)\chi_D$ for some $\si_c >0$. More precisely, we consider the problem given by adding an overdetermined condition of Serrin-type to \eqref{eq:2phase problem}. That is, we focus on the following overdetermined problem: 
\begin{equation}\label{odp}
    \begin{cases}
    -\dv\left(\sg\gr u\right)=1 \, \mbox{ in } \, \Om,\\
    u=0 \, \mbox{ on } \, \pa\Om,\\
    \pa_n u=c \, \mbox{ on } \, \pa\Om, 
    \end{cases}
\end{equation}
where $n$ denotes the outward unit normal vector of $\pa \Omega$ and $\pa_n$ is the corresponding normal derivative. 
By integration by parts, it is easy to see that, if the overdetermined problem \eqref{odp} is solvable, then the parameter $c$ must be given by 
\begin{equation}\label{eq:value of c}
    c = - \frac{|\Omega|}{|\pa \Omega|}. 
\end{equation}

There are two different approaches for studying the solutions $(D,\Om)$ of the overdetermined problem above. Indeed, the overdetermined problem \eqref{odp} can be either regarded as an ``inner problem" or as an ``outer problem". 
Roughly speaking, the outer problem consists in determining the domain $\Om$ given $D$, while the inner problem consists in determining the inclusion $D$ given $\Om$ (for a precise definition of the inner problem and outer problem, see \cite{CY2019}). 


\begin{figure}[h]
\centering
\includegraphics[width=0.45\linewidth]{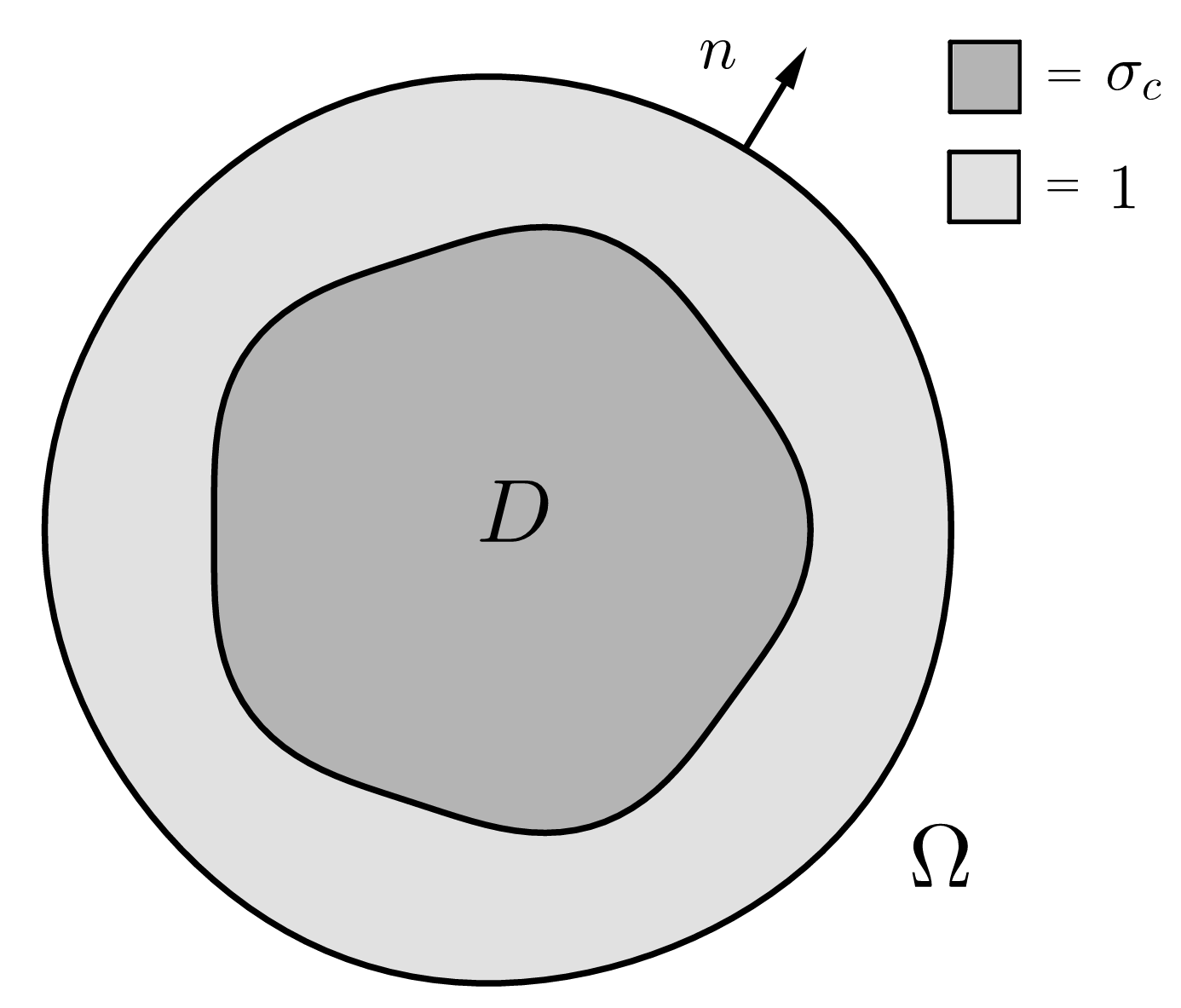}
\caption{Problem setting} 
\label{nonlinear}
\end{figure}

When $\si_c = 1$ (or, equivalently, $D = \emptyset$), it is known from Serrin's paper \cite{Se1971} that the overdetermined problem \eqref{odp} is solvable if and only if the domain $\Omega$ is a ball. 
%
%
In this paper, we will refer to the original Serrin's overdetermined problem as the ``one-phase problem". 

The two-phase setting, that is, when $\si_c \neq 1$ and $D \neq \emptyset$, is more complicated since solutions of the overdetermined problem \eqref{odp} are affected by the geometry of the inclusion $D$ or the domain $\Omega$. The first author and the third author, in \cite{CY2019}, proved local existence and uniqueness for the outer problem near concentric balls under some non-criticality condition on the coefficients and then gave a numerical algorithm for finding the solutions to the outer problem based on the Kohn--Vogelius functional and the augmented Lagrangian method. Furthermore, in \cite{CY2020}, they proved that there exist symmetry-breaking solutions of \eqref{odp} for certain critical values of $\si_c$. Similar problems involving two-phase conductors have been studied in several situations. We refer to \cite{MT1997A, MT1997B, CMS2009, CLM2012, L2014, CSU2019, Ca2020, camasa, Ca2021}. 

Let $(D,\Om)$ denote a solution of the overdetermined problem \eqref{odp}. One would expect that, if either $\sg_c \simeq 1$ or $D$ is small enough in some sense, then $\Om$ must be close to a ball (the solution of the one-phase problem). This was conjectured in the paper \cite{CY2019} from the numerical results. The purpose of this paper is to give quantitative stability estimates that show how close the solution $\Om$ is to a ball when either $\sg_c \simeq 1$ or $|D|$ is small.  
\begin{figure}[htbp]
\begin{minipage}{0.5\hsize}
\begin{center}
\includegraphics[width=50mm]{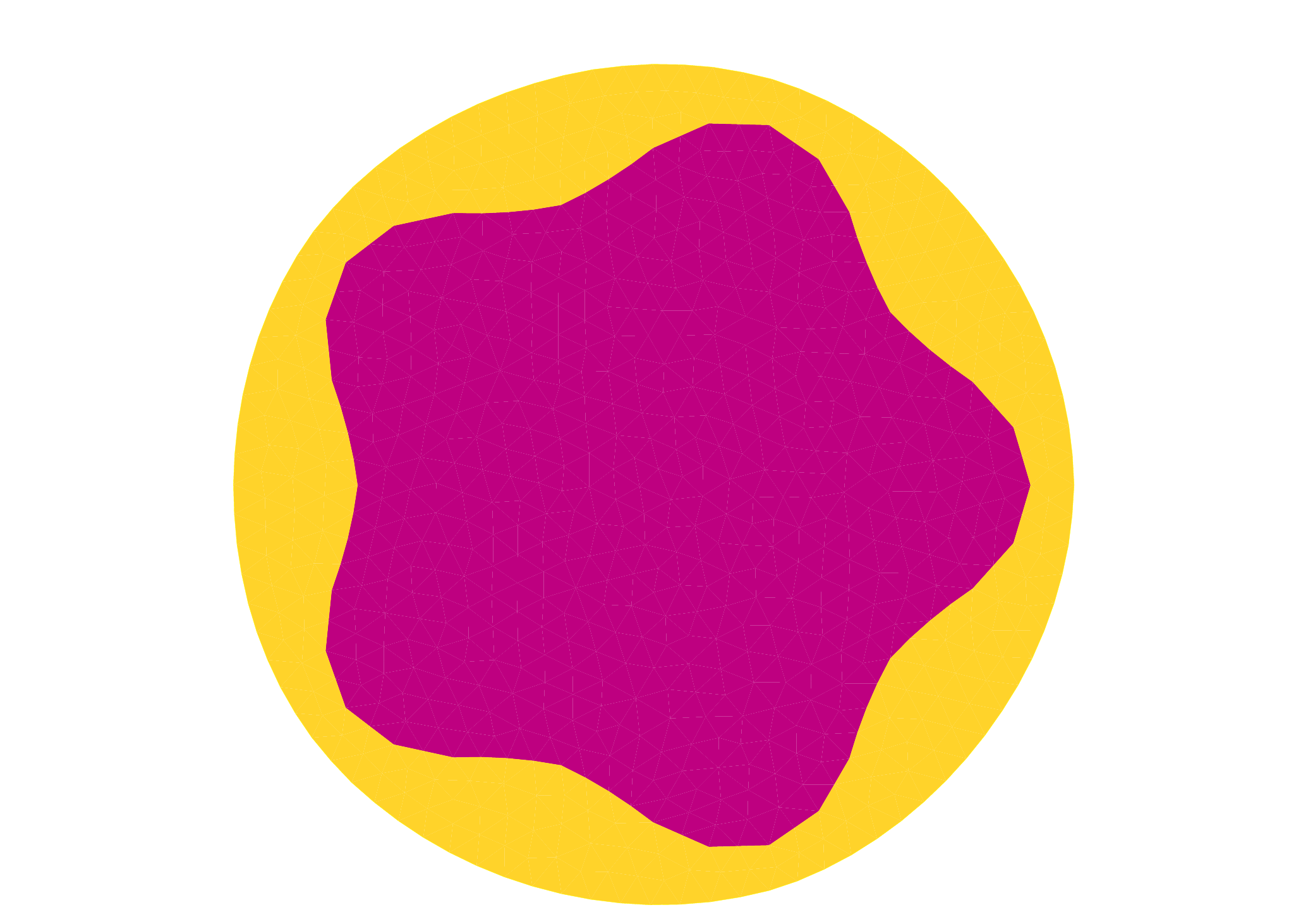}
\end{center}
\end{minipage}
\begin{minipage}{0.5\hsize}
\begin{center}
\includegraphics[width=50mm]{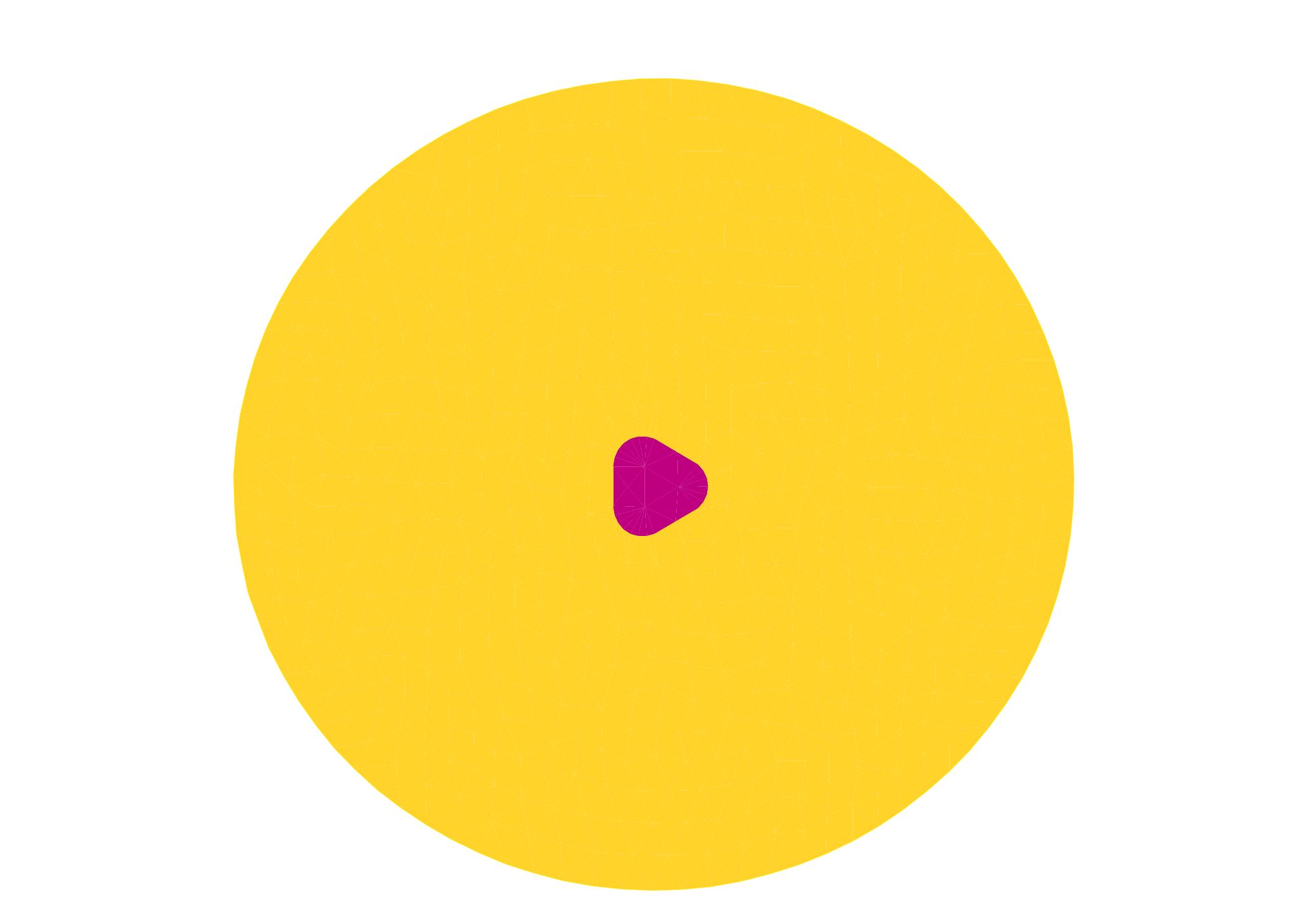}
\end{center}
\end{minipage}
\begin{minipage}{0.5\hsize}
\caption{Numerical result when $\sg_c \simeq 1$}
\label{fig:2}
\end{minipage}
\begin{minipage}{0.5\hsize}
\caption{Numerical result when $|D|$ is small}
\label{fig:3}
\end{minipage}
\end{figure}

We begin by setting some relevant notations. The diameter of $\Om$ is indicated by $d_\Om$. For a point $z\in\Om$, $\rho_i$ and $\rho_e$ will denote 
the radius of the largest ball contained in $\Om$
and that of the smallest ball that contains $\Om$, both centered at $z$ (see Figure \ref{fig:rhoirhoe}); in formulas, 
\begin{equation}
\label{def-rhos}
\rho_i=\min_{x\in \pa \Om }|x-z|  \ \mbox{ and } \ \rho_e=\max_{x\in \pa \Om }|x-z|.
\end{equation}
In what follows, the point $z$ will be always taken as later specified in Theorem \ref{thm:Improved-Serrin-stability}.

%
%

%
%

\begin{figure}[htbp]
\centering
\includegraphics[width=65mm]{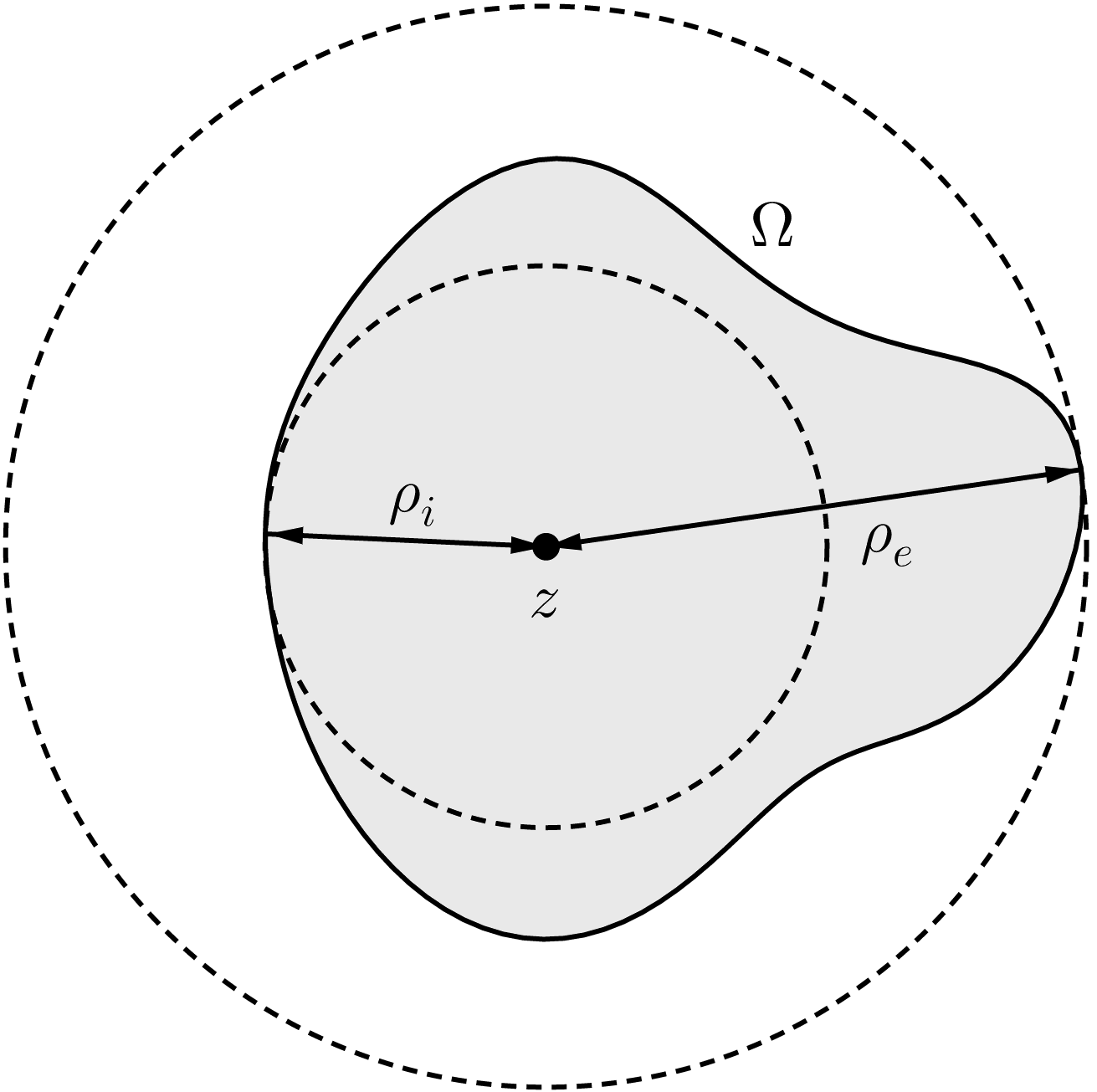}
\caption{$\rho_i$ and $\rho_e$.}
\label{fig:rhoirhoe}
\end{figure}

If $\pa\Om$ is of class $C^{1,\al}$ (see \cite[p.94]{GT} for a definition), 
then from the compactness of $\pa\Om$, there exist two positive constants $K$ and $\rho_0$ such that for all
$x \in \pa\Om$ and $0 < \rho \le  \rho_0$ there exists $x_0 \in\pa\Om$ and a one-to-one mapping $\Psi$ of $B_\rho(x_0)$ onto
$\om \subset \rn$ such that $x \in B_\rho(x_0)$ and
\begin{equation*}
\begin{aligned}
&\Psi \left(B_\rho(x_0) \cap \Om\right)\subset \{x_N>0\}, \quad  &\Psi\left(B_\rho(x_0) \cap \pa\Om\right)\subset \{x_N=0\},\\
&\norm{\Psi}_{C^{1,\al}(B_\rho(x_0))} \le K, \quad
&\norm{\Psi^{-1}}_{C^{1,\al}(\om)}\le K.   
\end{aligned}    
\end{equation*}
We will refer to the pair $(K,\rho_0)$ as the $C^{1,\al}$ modulus of $\pa\Om$ (see also \cite{ABR, BNST2008} for a similar definition in the case of $C^{2,\al}$ domains and \cite{LiVogelius} for another definition of the $C^{1,\al}$ modulus). 

%
%
%
In what follows, we state the main theorems of this paper. 
The following stability result for the one-phase problem will be crucial to establish quantitative stability estimates of the two-phase overdetermined problem \eqref{odp}. 

\begin{mainthm}[Stability for the one-phase problem with $L^2$ deviation in terms of the $C^{1,\al}$ modulus of $\pa \Om$]
\label{thm:Improved-Serrin-stability}
Let $\Om\subset\RR^N$ be a bounded domain with boundary $\pa \Om$ of class $C^{1, \al}$ and let $c$ be the constant defined in \eqref{eq:value of c}.
Let $v$ be the solution of \eqref{eq:2phase problem} with $\sigma_c = 1$ and 
let $z\in\Om$ be a point such that $v(z)=\displaystyle\max_{\ol\Om} v$.
%
%
Then, there exists a positive constant $C_1$ such that
\begin{equation}
\label{general improved stability serrin C}
\rho_e-\rho_i\le C_1 \,\nr \pa_n v - c \nr_{ L^2 ( \pa \Om ) }^{\tau_N} ,
\end{equation}
with the following specifications:
%
%
%
\begin{enumerate}[(i)]
\item $\tau_2 = 1$;
\item $\tau_3$ is arbitrarily close to one, in the sense that for any $\theta>0$, there exists a positive constant $C_1$ such that  \eqref{general improved stability serrin C} holds with $\tau_3 = 1- \theta$;
\item $\tau_N = 2/(N-1)$ for $N \ge 4$.
\end{enumerate}

The constant $C_1$ depends on $N$, $d_\Om$, the $C^{1,\al}$ modulus of $\pa \Om$,
and $\theta$ (only in the case $N=3$).
%
%
\end{mainthm}

\begin{remark}
The proof of Theorem \ref{thm:Improved-Serrin-stability} relies on (and is hugely an adaptation of) the techniques developed by Magnanini and the second author in \cite{PogTesi, MP2, MP3}.
When the $C^{1, \al}$ modulus of $\pa \Om$ is replaced by the uniform interior and exterior touching ball condition, Theorem \ref{thm:Improved-Serrin-stability} 
%
%
is contained in \cite{PogTesi,MP3}.
We point out that Theorem \ref{thm:Improved-Serrin-stability} provides a new extension of
\cite[Theorem 3.1]{MP3} in which the constant $C_1$ appearing in \eqref{general improved stability serrin C} depends on the $C^{1, \al}$ modulus of $\pa \Om$ instead of the radii of the uniform interior and exterior touching ball condition (as it happened in \cite{MP3}). We stress that the uniform interior and exterior touching ball condition is equivalent to the $C^{1,1}$ regularity of $\pa \Om$ (see, for instance, \cite[Theorem 1.0.9]{Barb} or \cite[Corollary 3.14]{ABMMZ}). The weaker $C^{1,\al}$ (with $0< \al <1$) regularity that we are considering here, is equivalent to a uniform interior and exterior touching pseudoball condition (see \cite[Theorem 1.3 and Corollary 3.14]{ABMMZ}).
\end{remark}

Thanks to Theorem \ref{thm:Improved-Serrin-stability},
%
%
we can obtain quantitative stability estimates for the two-phase overdetermined problem \eqref{odp} when $\sg_c \simeq 1$ and $|D|$ is small. 

\begin{mainthm}[Stability for $\sg_c \simeq 1$]\label{thm I}
Let $\Om\subset\RR^N$ be a bounded domain 
%
%
and let $D$ be an open set satisfying $\ol D\subset\Om$. Moreover, suppose that the pair $(D,\Om)$ is a solution to the overdetermined problem \eqref{odp}. 
%
%
Then,
we have that
\begin{equation*}
\rho_e-\rho_i\le C_2 |\sg_c-1|^{\tau_N},   
\end{equation*}
where $\tau_N$ is defined as in Theorem \ref{thm:Improved-Serrin-stability} and the constant $C_2 > 0$ depends on $N$, $d_\Om$, the $C^{1,\al}$ modulus of the boundary $\pa\Om$, and $\theta$ (only in the case $N=3$).
%
%
\end{mainthm}

\begin{mainthm}[Stability for $|D|$ small]\label{thm II}
Let $\Om\subset\RR^N$ be a bounded domain 
%
%
and let $D$ be an open set satisfying $\ol D\subset\Om$. Moreover, suppose that the pair $(D,\Om)$ is a solution to the overdetermined problem \eqref{odp}. 
%
%
Then, 
we have that
\begin{equation*}
\rho_e-\rho_i\le C_3 |D|^{\frac{\tau_N}{2}},   
\end{equation*}
where $\tau_N$ is defined as in Theorem \ref{thm:Improved-Serrin-stability} and the constant $C_3 > 0$ depends on $N$, $d_\Om$, $\sg_c$, the $C^{1,\al}$ modulus of the boundary $\pa\Om$, the distance between $\ol D$ and $\pa\Om$, and $\theta$ (only in the case $N=3$).
\end{mainthm}

\begin{remark}[On the regularity]
Even without imposing any regularity assumptions on $\pa \Om$ (in Theorems \ref{thm I} and \ref{thm II}), \cite[Theorem 1]{Vo} guarantees that if $u$ satisfies \eqref{odp} (where the boundary conditions are interpreted in the appropriate weak sense), then $\pa \Om$ is of class $C^{2,\ga}$, with~$0 < \ga < 1$.
In particular, the $C^{1, \al}$ modulus of $\pa \Om$ is well defined, and the notation~$\pa_n u=c$ on $\pa \Om$ is well posed in the classical sense. 
Furthermore, the regularity of $\pa\Om$ can be bootstrapped even more. Indeed, once one knows that $(D,\Om)$ is a classical solution of \eqref{odp}, then the local result \cite[Theorem 2]{Kinderlehrer Nirenberg} implies that $\pa \Om$ must be an analytic surface. 
\end{remark}

\begin{remark}
Theorem \ref{thm II} should be compared with the results obtained (with a different approach) by Dipierro, Valdinoci, and the second author in \cite{DPV}. Although the results in \cite{DPV} apply to the more general setting in which the equation is not known (and could be arbitrary) in $D$, in the case of the two-phase problem \eqref{odp} considered here, Theorem \ref{thm II} provides substantial improvements. First, in \cite{DPV} the closeness of $\Om$ to a ball is controlled by $|\pa D |$, while Theorem \ref{thm II} provides a stronger control in terms of $|D|$. Also, the constant $C$ appearing in the estimates in \cite{DPV} also depends on the $C^{2}$ norm of $u$ on $\pa D$, and that dependence does not appear in Theorem \ref{thm II}. We mention that, in the present setting, such regularity of $u$ up to $\pa D$ would be available at the cost of assuming some regularity of $\pa D$ (see \cite{XB2013}), which is not assumed in Theorem \ref{thm II}.
\end{remark}

%
%
From Theorem \ref{thm I} and \ref{thm II}, we can show the non-existence for the inner problem of the two-phase overdetermined problem \eqref{odp} when $\sg_c \simeq 1$ and $|D|$ is small. 
\begin{maincorol}[Non-existence for $\sg_c \simeq 1$]\label{thm III}
Let $\Om\subset\rn$ be a bounded domain 
and suppose that $\Om$ is not a ball (that is, $\rho_e-\rho_i>0$). Then, the overdetermined problem \eqref{odp} does not admit a solution of the form $(D,\Om)$ if \begin{equation*}
    |\sg_c-1| < 
    C_4\left( {\rho_e-\rho_i} \right)^{\frac{1}{\tau_N}}, 
\end{equation*}
where $\tau_N$ is defined as in Theorem \ref{thm:Improved-Serrin-stability} and the constant $C_4$ can be explicitly written as 
\begin{equation*}
C_4=\left(C_2\right)^{-1/\tau_N},
\end{equation*}
where $C_2$ is the constant that appears in the statement of Theorem \ref{thm I}.
\end{maincorol}

\begin{maincorol}[Non-existence for $|D|$ small]\label{thm IV}
Let $\Om\subset\rn$ be a bounded domain 
and suppose that $\Om$ is not a ball (that is, $\rho_e-\rho_i>0$). Then, the overdetermined problem \eqref{odp} does not admit a solution of the form $(D,\Om)$ if \begin{equation*}
    |D| < 
    C_5\left( {\rho_e-\rho_i} \right)^{\frac{2}{\tau_N}}. 
\end{equation*}
where $\tau_N$ is defined as in Theorem \ref{thm:Improved-Serrin-stability} and the constant $C_5$ can be explicitly written as 
\begin{equation*}
C_5=\left(C_3\right)^{-2/\tau_N},
\end{equation*}
where $C_3$ is the constant that appears in the statement of Theorem \ref{thm II}.
\end{maincorol}

This paper is organized as follows. In Section \ref{pre}, we provide stability results for the one-phase problem and prove Theorem \ref{thm:Improved-Serrin-stability}. Section \ref{sec:thm I} is devoted to the proof of Theorem \ref{thm I} by the implicit function theorem for Banach spaces and a corollary of Theorem \ref{thm:Improved-Serrin-stability}. In Section \ref{sec:thm II}, we prove Theorem \ref{thm II} by a perturbation argument using Green's function of the Dirichlet boundary value problem for the Laplace operator and a corollary of Theorem \ref{thm:Improved-Serrin-stability}. In Section \ref{sec:nonexistence}, we show the non-existence for the inner problem of the two-phase overdetermined problem \eqref{odp} from Theorems \ref{thm I} and \ref{thm II}.


\section{Proof of Theorem \ref{thm:Improved-Serrin-stability}}\label{pre}
In this section, we consider $v$ solution of \eqref{eq:2phase problem} with $\sigma_c = 1$, that is,
\begin{equation}\label{serrin1}
- \De v = 1 \quad \text{ in } \Om, \quad v = 0 \quad \text{ on } \pa \Om .
\end{equation}
The stability issue for the classical Serrin's problem has been deeply studied by several authors in \cite{ABR, BNST2008, CMV, Fe, MP, Pog, PogTesi, MP2, MP3, gilsbach onodera, MP6NewInterpolating}. 
A more detailed overview and comparison of those results can be found in \cite{Mag, PogTesi, MP2, MP3}.

We now give the proof of Theorem \ref{thm:Improved-Serrin-stability}. 
\begin{proof}[Proof of Theorem \ref{thm:Improved-Serrin-stability}]
As already mentioned, the result with the $C^{1, \al}$ modulus of $\pa \Om$ replaced by the uniform interior and exterior touching ball condition has been obtained in \cite{PogTesi,MP3}. 
Here, we hugely exploit tools and techniques developed in \cite{PogTesi,MP3}, adapting them to our (more general) setting. More precisely, we are going to point out how to modify the proof of \cite[Theorem 3.1]{MP3} in the present setting, referring the reader to \cite{PogTesi,MP3} for the remaining details.
%
%

In this proof, we use the letter $C$ to denote a positive constant whose value could change by line to line; the parameters on which $C$ depends will be specified each time. The letter $c$ will always indicate the constant in \eqref{eq:value of c}.
 
{\it Step 1} (Fundamental identity). 
By following \cite{MP3} and taking into account that here a different normalization of \eqref{serrin1} is adopted, we introduce the function
$q(x)= - \frac{|x-z|^2}{2 N}$ (where $z$ is a global maximum point of $v$ in $\Om$) and the harmonic function $h = v-q$.
In the present setting, Identity $(3.1)$ in \cite{MP3} reads
\begin{equation}\label{eq:FI Classical Serrin}
	\int_{\Om} v \, |\na^2 h|^2  \, dx = \frac{1}{2} \int_{\pa \Om} \left(c^2 - (\pa_n v)^2 \right) \pa_n h \, dS_x ,  
\end{equation}
where $c$ is the constant given by \eqref{eq:value of c}.

Notice that, by definition $h$ is harmonic and, being $h=-q$ on $\pa \Om$, we have that
$$\osc_{\pa\Om} h := \max_{\pa \Om} h - \max_{\pa \Om} h = \frac{\rho_e^2 -\rho_i^2}{2N} .$$
This last relation and the inequality $\rho_e +\rho_i \ge \rho_e \ge d_\Om/2$ immediately lead to
\begin{equation}\label{eq:oscillation and rho}
\rho_e-\rho_i \le \frac{4N}{d_\Om} \osc_{\pa \Om} h.  	
\end{equation}

{\it Step 2} (Optimal growth of $v$ from the boundary). We prove that
\begin{equation}\label{eq:optimal growth}
	v(x) \ge C \, \de_{\pa \Om}(x) \quad \text{ for any } x \in \ol{ \Om } ,
\end{equation}
where $\de_{\pa \Om}(x):= \dist(x, \pa \Om)$ denotes the distance function to $\pa \Om$, and $C$ is a constant only depending on $N$ and the $C^{1,\al}$ modulus of $\pa \Om$.

By the Hopf-Olenik lemma for $C^{1,\al}$ domains\footnote{Hopf-Olenik Lemma for $C^{1,\al}$ domains is due to Giraud \cite{Giraud}. We refer to \cite[Section 4.1]{ABMMZ} for a historical perspective on this subject.} (see, for instance, the more general version contained in \cite[Theorem 4.4]{ABMMZ}), for any $x_0 \in \pa \Om$ we have that 
\begin{equation}\label{eq:Hopf-Olenik for C1alpha}
	v(x_0 - t \, n) \ge k \, t \quad \text{ for any }  0 < t < \de ,	
\end{equation}
where $k$ and $\de$ are two constants only depending on $N$ and the $C^{1,\al}$ modulus of $\pa \Om$.
This, together with the rough estimate
\begin{equation}\label{eq:rough growth}
v(x) \ge \frac{\de_{\pa \Om}(x)^2}{2N} \quad \text{ for any } x \in \ol{ \Om },
\end{equation}
easily leads to the global inequality \eqref{eq:optimal growth} with $C= \max\left\lbrace k, \frac{\de}{2N} \right\rbrace$, where $k$ and $\de$ are those in \eqref{eq:Hopf-Olenik for C1alpha}. For a proof of \eqref{eq:rough growth} see, for instance, the first claim in \cite[Lemma 3.1]{MP2}.

{\it Step 3} (Key inequality). Here, we prove that
\begin{equation}\label{eq:key inequality rho and derivate seconde}
	\rho_e -\rho_i \le C \nr \de_{\pa \Om}^{1/2} \na^2 h \nr_{L^2( \Om )}^{\tau_N}, 
\end{equation}
where $\tau_N$ is as in the statement of Theorem \ref{thm:Improved-Serrin-stability}, and $C$ only depends on $N$, $d_\Om$, the $C^{1,\al}$ modulus of $\pa \Om$, and $\theta$ (only in the case $N=3$).

The reference result here is \cite[Theorem 2.8]{MP3}. 
To extend \cite[Theorem 2.8]{MP3} in the present setting, we need an appropriate extension of \cite[Lemma 2.7]{MP3}, which is provided
%
%
in \cite{MP4Interpolating}.
Here, it is enough to apply \cite[Theorem 3.1]{MP4Interpolating} with $L=\De$, $v=h$, $\al=1$ to get that
\begin{equation}\label{eq:lemma con interpolating e norma grad}
\osc_{\pa \Om} h \le C \nr \na h \nr_{L^\infty(\Om)}^{N/(N+p)} \nr h - h_\Om \nr_{L^p(\Om)}^{p/(N+p)} ,
\end{equation}
where $h_\Om$ denotes the mean value of $h$ on $\Om$ and $C$ only depends on $N$, $p$, $d_\Om$, and the $C^{1,\al}$ modulus\footnote{\cite[Theorem 3.1]{MP4Interpolating} has been proved for domains satisfying a uniform interior cone condition. This class of domains contains that of Lipschitz domains, which in turn contains $C^{1,\al}$ domains. Of course, the parameters of the uniform interior cone condition appearing in the estimate in \cite[Theorem 3.1]{MP4Interpolating} can be bounded in terms of the $C^{1,\al}$ modulus of $\pa \Om$.} of $\pa \Om$. 
Notice that, $\nr \na h \nr_{L^\infty(\Om)}$ can be estimated in terms of $N$, $d_\Om$, and the $C^{1, \al}$ modulus of $\pa \Om$, by putting together
$$
\nr \na h \nr_{L^\infty(\Om)} = \max_{\ol{\Om}} |\na h| \le \max_{\ol{\Om}} |\na v| + \frac{d_\Om}{N}
$$
and the classical Schauder estimate for $\max_{\ol{\Om}} |\na v|$.
%
%
Thus, \eqref{eq:oscillation and rho} and \eqref{eq:lemma con interpolating e norma grad} ensure that
\begin{equation}\label{eq: stima rho and Lp}
\rho_e -\rho_i \le C  \nr h - h_\Om \nr_{L^p(\Om)}^{p/(N+p)} 
\end{equation}
holds true with a constant $C$ only depending on $N$, $p$, $d_\Om$, and the $C^{1,\al}$ modulus of $\pa \Om$.

With this at hand, one can directly check that replacing \cite[Equation (1.13)]{MP3} and \cite[Lemma 2.7]{MP3} with \eqref{eq:oscillation and rho} and \eqref{eq: stima rho and Lp} in the proof of \cite[Theorem 2.8]{MP3} leads to\footnote{ 
To this end, one must check that the constants appearing in the weighted Poincar\'e-type inequalities \cite[Equation (2.8) and item (i) of Corollary 2.3]{MP3} 
%
%
(applied to $h$) and
%
%
in the Morrey-Sobolev-type inequality \cite[Equation (2.20)]{MP3} can indeed be bounded in terms of the the above mentioned parameters.

As stated in \cite[Lemma 2.1 and Corollary 2.3]{MP3}, the Poincar\'e-type inequalities in \cite[Equation (2.8) and item (i) of Corollary 2.3]{MP3} hold true in the huge class of John domains (see \cite{PogTesi} and references therein for more details), which in particular contains $C^{1,\al}$ domains. 
Moreover, \cite[items (i) and (ii) of Remark 2.4]{MP3} give explicit estimates for the constants in \cite[Equation (2.8) and item (i) of Corollary 2.3]{MP3} in terms of $d_\Om$, $\de_{\pa \Om}(z)$, and the so-called John parameter of $\Om$.
Now, the John parameter can be bounded in terms of the $C^{1,\al}$ modulus of $\pa \Om$ and $d_\Om$.
Also, the dependency on $\de_{\pa \Om}(z)$ can be dropped thanks to the inequality
\cite[Equation (2.21)]{MP3} with $M=\max_{ \ol{\Om}} |\na v|$ and the radius $r_i$ replaced by the inradius $r_\Om$, i.e., the radius of any largest ball contained in $\Om$ (see also \cite{MP5Hotspots}).  
In turn, both $r_\Om$ and $\max_{ \ol{\Om}} |\na v|$ can be estimated by the $C^{1,\al}$ modulus of $\pa \Om$.

Finally, the constant in the Morrey-Sobolev-type inequality \cite[Equation (2.20)]{MP3} only depends on the parameters of a uniform interior cone condition (see \cite[Remark 2.9]{MP3} and
\cite[Theorem~9.1]{Fr1983}), which can be easily bounded in terms of the $C^{1,\al}$ modulus of $\pa \Om$.

These observations complete the proof of \eqref{eq:key inequality rho and derivate seconde}.
An alternative approach toward \eqref{eq:key inequality rho and derivate seconde}, which also applies in the present setting, can be found in \cite{MP6NewInterpolating}.
}
\eqref{eq:key inequality rho and derivate seconde}.

{\it Step 4} (Final estimate for the left-hand side of \eqref{eq:FI Classical Serrin}). Putting together \eqref{eq:key inequality rho and derivate seconde} and \eqref{eq:optimal growth} immediately gives that
\begin{equation}\label{eq:LHSIDE FIdentity Estimate}
	\rho_e -\rho_i \le C \left( \int_{\Om} v \, |\na^2 h|^2  \, dx \right)^{\tau_N /2} ,
\end{equation}
where $\tau_N$ is as in the statement of Theorem \ref{thm:Improved-Serrin-stability}, and $C$ only depends on $N$, $d_\Om$, the $C^{1,\al}$ modulus of $\pa \Om$, and $\theta$ (only in the case $N=3$).

{\it Step 5} (Estimate for the right-hand side of \eqref{eq:FI Classical Serrin}).
We start by estimating from above the right-hand side of \eqref{eq:FI Classical Serrin} by using H\"older's inequality as follows:
\begin{equation}\label{eq:NEW CLASSICAL SERRIN step 1}
	\int_{\pa \Om} \left(c^2 - (\pa_n v)^2 \right) \pa_n h \, dS_x \le \left( c +  \max_{ \ol{ \Om }} | \na v | \right) \nr \pa_n v - c \nr_{L^2 (\pa \Om)} \nr \pa_n h \nr_{L^2 (\pa \Om)}.
\end{equation}
Notice that $c + \max_{\ol{\Om}} |\na v|$ can be bounded above by a constant depending on $N$, $d_\Om$, and the $C^{1,\al}$ modulus of $\pa \Om$; this easily follows in light of the classical Schauder estimates for $\max_{\ol{\Om}} |\na v|$, and estimating $|c|$ by putting together \eqref{eq:value of c},
the isoperimetric inequality
$$
| \pa \Om | \ge N | B_1 |^{1/N} | \Om |^{(N-1)/N} ,
$$
and the trivial bound
$$ 
|\Om| \le |B_1| (d_\Om/2)^N , 
$$
where $B_1$ denotes a unit ball in $\RR^N$.

Now, reasoning as in \cite[Lemma 2.5]{MP3}\footnote{We can repeat the proof of \cite[(i) of Lemma 2.5]{MP3} (with $u=v$ and $v=h$) just by replacing \cite[(1.15)]{MP3} with \eqref{eq:optimal growth} and \cite[Theorem 3.10]{MP} with 
	\begin{equation*}
		- \pa_n v \ge k, \quad \text{ where } k \text{ is the constant appearing in \eqref{eq:Hopf-Olenik for C1alpha}},
	\end{equation*}
which easily follows from \eqref{eq:Hopf-Olenik for C1alpha}. Also, we took into account that a different normalization of the problem \eqref{serrin1} was adopted in \cite{MP3}.
%
%
}, we can prove that
\begin{equation}\label{eq:trace inequality pan h}
	\nr \pa_n h \nr_{L^2 (\pa \Om)}^2 \le C \int_{\Om} v \, |\na^2 h|^2  \, dx ,
\end{equation}
where $C$ is a constant only depending on $N$, $d_{\Om}$, and the $C^{1, \al}$ modulus of $\pa \Om$.

As in the proof of \cite[Theorem 3.1]{MP3}, putting together \eqref{eq:FI Classical Serrin}, \eqref{eq:NEW CLASSICAL SERRIN step 1} and \eqref{eq:trace inequality pan h} gives that
\begin{equation}\label{eq:RHside stima pa n h}
	\nr \pa_n h \nr_{L^2 (\pa \Om)} \le C \,	\nr \pa_n v - c \nr_{L^2 (\pa \Om)} ,
\end{equation}
now with a constant $C$ only depending on $N$, $d_{\Om}$, and the $C^{1, \al}$ modulus of $\pa \Om$. Thus, combining \eqref{eq:FI Classical Serrin}, \eqref{eq:NEW CLASSICAL SERRIN step 1} and \eqref{eq:RHside stima pa n h} gives \begin{equation}\label{last inequality}
\int_{\Om} v \, |\na^2 h|^2  \, dx \le C \, \nr \pa_n v - c \nr_{L^2 (\pa \Om)}^2 ,    
\end{equation}
where $C$ is a constant only depending on $N$, $d_{\Om}$, and the $C^{1, \al}$ modulus of $\pa \Om$.
The conclusion of Theorem \ref{thm:Improved-Serrin-stability} immediately follows by combining \eqref{last inequality} and \eqref{eq:LHSIDE FIdentity Estimate}.
\end{proof}

We remark that, in the proofs of Theorems \ref{thm I} and \ref{thm II}, we are able to obtain an upper bound on the uniform norm of the deviation of $\pa_n v$ from $c$. We are therefore interested in an estimate similar to \eqref{general improved stability serrin C} but where the $L^2$ norm 
is replaced by the uniform norm. In other words, what we really need in the proofs of Theorems \ref{thm I} and \ref{thm II} is the following Corollary of Theorem \ref{thm:Improved-Serrin-stability}. Nevertheless, since Theorem \ref{thm:Improved-Serrin-stability} is of independent interest, we decided to state it in its full generality in the introduction of this paper.

\begin{corollary}[Stability for the one-phase problem with uniform deviation in terms of the $C^{1,\al}$ modulus of $\pa \Om$]
\label{cor:Improved-Serrin-stability-deviation in uniformnorm}
Let $\Om\subset\RR^N$ be a bounded domain with boundary $\pa \Om$ of class $C^{1, \al}$ and $c$ be the constant defined in \eqref{eq:value of c}.
Let $v$ be the solution of \eqref{serrin1}
%
%
and 
let $z\in\Om$ be a point such that $v(z)=\displaystyle\max_{\ol\Om} v$. Then there exists a positive constant $C_6$ such that
\begin{equation}
\label{general improved stability serrin C with uniform deviation}
\rho_e-\rho_i\le C_6\,\nr \pa_n v - c \nr_{L^\infty (\pa \Om) }^{\tau_N} ,
\end{equation}
where $\tau_N$ is defined as in Theorem \ref{thm:Improved-Serrin-stability} and the constant $C_6 > 0$ only depends on $N$, $d_\Om$, the $C^{1,\al}$ modulus of the boundary $\pa\Om$, and $\theta$ (only in the case $N=3$).
%
%
\end{corollary}
\begin{proof}
%
%
%
Since
\begin{equation}\label{eq:stima con perimetro da L2Serrin a Linfinito}
	\nr \pa_n v - c \nr_{L^2 (\pa \Om) } \le | \pa \Om |^{1/2} \, \nr \pa_n v - c \nr_{L^\infty (\pa \Om) }
\end{equation}
trivially holds true,
the desired result
can be easily deduced by \eqref{general improved stability serrin C}. It only remains to notice that we can get rid of the dependence on $| \pa \Om|$ appearing in \eqref{eq:stima con perimetro da L2Serrin a Linfinito}, thanks to the bound
$$
| \pa \Om | \le \frac{| \Om |}{k}, \quad \text{ where } k \text{ is the constant appearing in \eqref{eq:Hopf-Olenik for C1alpha}.}
$$
The last bound follows by putting together the identity
$$
| \Om | = \int_{\Om} (- \De v ) \, dx = \int_{\pa \Om} (- \pa_n v ) \, dS_x
$$
%
%
and the inequality $- \pa_n v \ge k$, which easily follows from \eqref{eq:Hopf-Olenik for C1alpha}.
\end{proof}

\section{Proof of Theorem \ref{thm I}}\label{sec:thm I}
In this section, we prove Theorem \ref{thm I}. First, we will show the Fr\'echet differentiability of the solution of \eqref{eq:2phase problem} with respect to the parameter $\sigma_c$. 
\begin{lemma}\label{frechet differentiability}
Let $\Om\subset\rn$ be a bounded domain of class $C^{1,\al}$ and $D$ be an open set 
such that $\ol D \subset \Om$. 
Moreover, let $U\subset \ol U\subset \Om$ be an open neighborhood of $\ol D$ of class $C^{1,\al}$.
For $t\in (-1,\infty)$, let $u(t)\in H_0^1(\Om)$ denote the solution of \eqref{eq:2phase problem} with respect to $\sg(t)=1+t\chi_D$ (that is, $\sg_c=1+t$). Then, $u(\cdot)$ defines a Fr\'echet differentiable map
\begin{equation*}
t\mapsto u(t)\in H_0^1(\Om)\cap 
C^{1,\al}(\ol \Om\setminus U).
\end{equation*}
Moreover, for every $t_0\in(-1,\infty)$, the Fr\'echet derivative $u'(t_0)$ is given by the solution of the following boundary value problem.
\begin{equation}\label{u'(t_0)}
    \begin{cases}
    -\dv \left( \sg(t_0)\gr u'(t_0) \right) =-\dv \left(\chi_D \gr u(t_0) \right) \quad \text{in }\Om,\\
    u'(t_0)=0 \quad \text{on }\pa\Om.
    \end{cases}
\end{equation}
\end{lemma}

The proof of Lemma \ref{frechet differentiability} relies on a standard method (see \cite[proof of Theorem 5.3.2, pp.206--207]{HP2005} for an application to shape-differentiability) based on the following implicit function theorem for Banach spaces (see \cite[Theorem 2.3, p.38]{AP1983} for a proof). 

\begin{theorem}
[Implicit function theorem]\label{ift}
Let $\Psi\in C^k(\La\times W,Y)$, $k\ge1$, where $Y$ is a Banach space and $\La$ (resp. $U$) is an open set of a Banach space $T$ (resp. $X$). Suppose that  $\Psi(\la^*,w^*)=0$ and that the partial derivative $\pa_w\Psi(\la^*,w^*)$ is a bounded invertible linear transformation from $X$ to $Y$. 

Then there exist neighborhoods $\Theta$ of $\la^*$ in $T$ and $W^*$ of $w^*$ in $X$, and a map $g\in C^k(\Theta,X)$ such that the following hold:
\begin{enumerate}[label=(\roman*)]
\item $\Psi(\la,g(\la))=0$ for all $\la\in\Theta$,
\item If $\Psi(\la,u)=0$ for some $(\la,u)\in\Theta\times U^*$, then $u=g(\la)$,
\item $g'(\la)=-\left(\pa_u \Psi(p) \right)^{-1}\circ \pa_\la \Psi(p)$, where $p=(\la,g(\la))$ and $\la\in\Theta$.
\end{enumerate}
\end{theorem}

\begin{proof}[Proof of Lemma \ref{frechet differentiability}]
For arbitrary $t\in(-1,\infty)$ and $u\in H_0^1(\Om)$, let $V(t,u)$ denote the solution to the following boundary value problem:
\begin{equation}\label{V(t,u)}
    \begin{cases}
    -\De V=-\dv\left(\sg(t)\gr u \right)-1 \quad \text{in }\Om,\\
    V=0\quad \text{on }\pa\Om.
    \end{cases}
\end{equation}
A functional analytical interpretation of this mapping is the following: we are identifying $V\in H_0^1(\Om)$ with the element $-\dv\left(\sg(t)\gr u \right)-1\in H^{-1}(\Om)$ whose action on $H_0^1(\Om)$ is defined via integration by parts, that is, for $\varphi\in H_0^1(\Om)$, 
\begin{equation*}
        \left(V,\varphi\right)_{H_0^1}= \int_\Om\gr V\cdot \gr\varphi=\int_{\Om}\sg(t)\gr u \cdot \gr \varphi-\int_\Om\varphi=\langle -\dv\left(\sg(t)\gr u \right)-1,\varphi\rangle.
\end{equation*}
By 
the classical Schauder estimates for the Dirichlet problem near the boundary (see, for instance \cite[Theorem 8.33]{GT} and the subsequent remarks) and the $L^\infty$ estimates \cite[Theorem 8.16]{GT}, we notice that $V(\cdot,\cdot)$ defines a mapping $(-1,\infty)\times X\to X$, where $X$ is the Banach space 
 \begin{equation*}
X:= H_0^1(\Om)\cap 
C^{1,\al}(\ol \Om\setminus U).
 \end{equation*}
By the defining properties of $V(t,u)$, it is clear that $u$ solves \eqref{eq:2phase problem} with $\sg=\sg(t)$ if and only if $V(t,u)\equiv 0$. In particular, for all $t_0\in (-1,\infty)$, the pair $(t_0,u(t_0))$ is a zero of $V$ by definition. 

We will now show that the map $V$ is (totally) Fr\'echet differentiable jointly in the variables $t$ and $u$. By the definition of $\sg(t)$ we can expand the left-hand side of \eqref{V(t,u)} as $-\De u- t\dv\left(\chi_D \gr u \right)-1$. By the linearity of problem \eqref{V(t,u)}, this implies that the map $V(t,u)$ can be decomposed as the sum of three parts:
\begin{equation*}
    V(t,u)=V_1(u)+V_2(t,u)+V_3,
\end{equation*}
where $V_i$ ($i=1,2,3$) are the solution of $-\De V= f_i$ with with Dirichlet zero boundary condition corresponding to 
\begin{equation*}
f_1= -\De u, \quad f_2= -t\dv\left(\chi_D \gr u \right), \quad f_3 =-1.    
\end{equation*}
Now, notice that, by construction, $V_1(u)$ is linear and continuous in $u$, $V_2(t,u)$ is bilinear and continuous in $(t,u)$ and $V_3$ does not depend on either $t$ or $u$. In particular, we get that $V_1$, $V_2$ and $V_3$ are all Fr\'echet differentiable. As a consequence, we get the Fr\'echet differentiability of the map $(t,u)\mapsto V(t,u)$ in the appropriate Banach spaces. Now, a simple computation yields that, for fixed $t_0\in(-1,\infty)$, the partial Fr\'echet differential $\pa_u V(t_0,u(t_0))$ is given by the mapping from the Banach space $X$ into itself defined as:
\begin{equation*}
 X\ni \varphi\mapsto    \pa_u V(t_0,u(t_0)) [\varphi]= W(t_0,\varphi), 
\end{equation*}
where $W(t_0,\varphi)\in X$ is the unique solution to the following boundary value problem:
\begin{equation}\label{W eq}
    \begin{cases}
    -\De W = -\dv\left( \sg(t_0) \gr \varphi\right) \quad \text{in }\Om,\\
    W=0\quad \text{on }\pa\Om.
    \end{cases}
\end{equation}
By ``inverting the roles" of the right and left-hand side in the above and applying once again 
the classical Schauder estimates for the Dirichlet problem near the boundary and the $L^\infty$ estimates as before, we can conclude that the map $\varphi\mapsto\pa_u V(t_0,u(t_0))[\varphi]$ is invertible (that is, problem \eqref{W eq} is well posed in the appropriate Banach spaces), as required. We can, therefore, apply the implicit function theorem to the map $(t,u)\mapsto V(t,u)$ at its zero $(t_0,u(t_0))$. This yields the existence of a Fr\'echet differentiable branch
\begin{equation*}
    (t_0-\ve,t_0+\ve)\ni t\mapsto \widetilde u(t)\in X\quad \text{such that }V(t,\widetilde u(t))=0.
\end{equation*}
In other words, $\widetilde u(t)$ also solves \eqref{eq:2phase problem}. Now, by the unique solvability of \eqref{eq:2phase problem}, $\widetilde u(t)=u(t)$, and therefore, the map $t\mapsto u(t)\in X$ is Fr\'echet differentiable, as claimed.
Finally, \eqref{u'(t_0)} is derived by simple differentiation with respect to $t$ of the weak form 
\begin{equation*}
\int_\Om \sg(t) \gr u(t)\cdot \gr \varphi =\int_\Om \varphi \quad \text{for all }\varphi\in H_0^1(\Om).
\end{equation*}
The proof is completed.
\end{proof}
\begin{proof}[Proof of Theorem \ref{thm I}]
As above, let $u(t)$ denote the solution to \eqref{eq:2phase problem} with $\sg=\sg(t)$. Moreover, suppose that, for some small 
$t_0\in(-1,1)$
, the function $u(t_0)$ satisfies the overdetermined condition
\begin{equation*}
\pa_n u(t_0)= c\quad \text{on }\pa\Om.
\end{equation*}
Consider the map 
\begin{equation}\label{t to pa_n u}
    (-1,\infty)\ni t\mapsto \restr{\pa_n u(t)}{\pa\Om}\in C^\al (\pa\Om). 
\end{equation}
Lemma \ref{frechet differentiability} tells us that the map defined by \eqref{t to pa_n u} is Fr\'echet differentiable. 
In particular, for all $x\in \pa\Om$, the map $t\mapsto\pa_n u(t)(x)\in\RR$ is differentiable.
By the fundamental theorem of calculus we have
\begin{equation*}
    \pa_n u(t_0)(x)-\pa_n u(0)(x)=\int_{0}^{t_0} \pa_n u'(\tau)(x)\ d\tau.
\end{equation*}
Therefore,
\begin{equation}\label{key estimate}
\norm{\pa_n u(t_0)-\pa_n u(0)}_{C^{\al}(\pa\Om)}\le |t_0| \max_I
\ \norm{\pa_n u'(\tau)}_{C^{\al}(\pa\Om)},
\end{equation}
where $I=\left[\min(0,t_0), \max(0,t_0)\right]$.
Again, by two applications of the classical Schauder estimates for the Dirichlet problem near the boundary \cite[Theorem 8.33]{GT} and the $L^\infty$ estimate \cite[Theorem 8.16]{GT}, we can estimate the right-hand side in the inequality above to get
\begin{equation}\label{v close to neumann}
    \norm{\pa_n v-c}_{L^\infty(\pa\Om)}
    \le\norm{\pa_n v-c}_{C^{\al}(\pa\Om)}
    =\norm{\pa_n u(t_0)-\pa_n u(0)}_{C^{\al}(\pa\Om)}
    \le C_7 |t_0|,
\end{equation}
where the constant $C_7>0$ depends only on $|\Om|$, $N$ 
and the $C^{1,\al}$ modulus of the boundary $\pa\Om$. 
By applying Corollary \ref{cor:Improved-Serrin-stability-deviation in uniformnorm}, we get the following estimate: 
\begin{equation*}
\rho_e-\rho_i\le C_2 |\sg_c-1|^{\tau_N},   
\end{equation*}
where $\tau_N$ is defined as in Theorem \ref{thm:Improved-Serrin-stability} and the constant $C_2$ depends on $N$, $d_\Om$, the $C^{1,\al}$ modulus of 
%
%
$\pa\Om$, and $\theta$ (only in the case $N=3$).
This is the desired estimate.
\end{proof}

\begin{remark}\label{extension of thm I}
The result of Theorem \ref{thm I} can be immediately extended to the case where 
$\pa \Om$ is of class $C^{1,\al}$ and
the overdetermined condition in \eqref{odp} reads \begin{equation}\label{extended odc}
\pa_n u(x)=c+\eta(x) \quad \text{for }x\in \pa\Om, 
\end{equation}
where the function $\eta\in L^\infty(\pa\Om)$ has vanishing mean over $\pa\Om$. Instead of \eqref{v close to neumann} we get 
\begin{equation}\label{pn v close to eta}
\norm{\pa_n v-c}_{L^\infty(\pa\Om)}\le 
\norm{\pa_n v-\pa_n u}_{L^\infty(\pa\Om)} + \norm{\pa_n u -c}_{L^\infty(\pa\Om)}\le
C_7|\sg_c-1|+ \norm{\eta}_{L^\infty(\pa\Om)}.
\end{equation}
Now, by applying Corollary \ref{cor:Improved-Serrin-stability-deviation in uniformnorm} we get the following estimates: 
\begin{equation}\label{Thm I remark}
\rho_e-\rho_i\le C_6 \left( C_7 |\sg_c-1|+\norm{\eta}_{L^\infty(\pa\Om)}  \right)^{\tau_N}.    
\end{equation}
\end{remark}

\section{Proof of Theorem \ref{thm II}}\label{sec:thm II}
In this section, we prove Theorem \ref{thm II}. Let $\Omega \subset \rn$ be a bounded domain of class $C^{1,\al}$ and $D$ be an open set such that $\ol D\subset\Om$. We also assume that $D$ satisfies 
\begin{equation}\label{assump:dist}
    \dist(D,\pa \Omega) \geq \dfrac{1}{M}, 
\end{equation}
where 
$M$ is a positive constant that for simplicity will be taken to be greater than $1$. 
Let us put $w = u - v$, where $u, v$ are the solutions of \eqref{eq:2phase problem} and \eqref{serrin1}, respectively. The function $w$ satisfies the following boundary value problem:
\begin{equation}\label{eq:diff}
\begin{cases}
-\Delta w = \dv \left( (\si_c - 1)\chi_D \nabla u \right) \, \mbox{ in } \, \Om,\\
w=0 \, \mbox{ on } \, \pa \Om. 
\end{cases}
\end{equation}
We consider a perturbation argument by using Green's function. 
Let $G(x,y)$ be the Green's function of the Dirichlet boundary value problem for the Laplace operator in $\Omega$. By \cite[pp.17--19]{GT}, the Green's function $G$ is represented by 
\begin{equation*}
    G(x,y) = \Gamma(x-y) - h(x,y), 
\end{equation*}
where $\Gamma$, defined for $x \in \RR^{N}\setminus\{0\}$, is the fundamental solution of Laplace's equation:
\begin{equation}\label{fundamental sol}
\Gamma(x) = 
\begin{cases}
     -\dfrac{1}{2\pi} \log|x| \quad (N = 2),\\
     \dfrac{1}{N(N-2)\om_N}\dfrac{1}{|x|^{N-2}}  \quad (N \geq 3), 
    \end{cases}
\end{equation}
(here $\om_N$ denotes the volume of the unit ball in $\rn$) and for $y \in \Om$, $h(\cdot, y)$ is the solution to the following Dirichlet boundary value problem:
\begin{equation}\label{eq:h}
    \begin{cases}
    -\Delta_x h(x,y) = 0 \quad x \in \Om,\\
    h(x,y)= \Gamma(x-y) \quad x \in \pa \Om. 
    \end{cases}
\end{equation}
The following gradient estimate for Green's function $G$ will be useful in the proof of Theorem \ref{thm II}. 
\begin{lemma}\label{esti for green}
Let $U:=\setbld{x\in\ol\Om}{\dist(x,D)>\frac{1}{2M}}$. Then, there exists a positive constant $C^*$ depending on $N$, $|\Om|$ and the $C^{1,\al}$ modulus of $\pa\Om$ such that 
\begin{equation*}
\sup_{(x,y)\in U\times D} |\gr_x \gr_y G(x,y)|\le C^* M^{N+1}.    
\end{equation*}
\end{lemma}
\begin{proof}
Fix $(x,y)\in U\times D$ and let $\be$ be a multi-index with $|\be|\ge1$.
From the definition of the fundamental solution \eqref{fundamental sol}, by direct calculation we obtain the estimate
\begin{equation}\label{esti:Gamma}
    |D^\be \Gamma(x-y)| \leq \dfrac{C(N)}{|x-y|^{N-2+|\be|}}\le C(N)M^{N-2+|\be|}, 
\end{equation}
where $C(N)>0$ is a constant depending only on $N$. 

In what follows, let us show the gradient estimate for $h$. First, notice that the function $\pa_{y_j}h(\cdot, y)$ is harmonic  on $\Om$ and verifies 
\begin{equation*}
\pa_{y_j} h(z,y) = \pa_{y_j} \Ga(z-y) \quad \text{for }z\in\pa\Om.     
\end{equation*}
Now, by the classical Schauder estimates for the Dirichlet problem near the boundary \cite[Theorem 8.33]{GT} and the $L^\infty$ estimate \cite[Theorem 8.16]{GT}, we can estimate $|\pa_{x_i}\pa_{y_j}h(x,y)|$ to get
\begin{equation*}
|\pa_{x_i}\pa_{y_j}h(x,y)|\le 
\norm{\pa_{y_j}h(\cdot, y)}_{C^{1,\al}(\ol\Om)}\le
C^* \norm{\pa_{y_j} \Ga(\cdot - y)}_{C^{1,\al}(\pa\Om)},    
\end{equation*}
where the constant $C^*>0$ only depends on 
$|\Om|$, $N$ and the $C^{1,\al}$ modulus of $\pa\Om$.
Now, up to redefining $C^*$, one can estimate the right-hand side in the above as follows
\begin{equation*}
C^* \norm{\pa_{y_j} \Ga(\cdot - y)}_{C^{1,\al}(\pa\Om)}\le C^* \norm{\Ga(\cdot-\cdot)}_{C^3(\pa\Om\times \ol D)} \le C^* M^{N+1},    
\end{equation*}
where we made use of \eqref{esti:Gamma} in the last inequality.
\end{proof}

\begin{proof}[Proof of Theorem \ref{thm II}]
Let us consider the Dirichlet boundary value problem \eqref{eq:diff}. For any $x\in U\cap \Om$, Green's representation formula gives us 
\begin{align*}
    w(x) &= \int_{\Omega} G(x,y) \mathrm{div}_{y} \left( (\si_c - 1)\chi_D \nabla u(y) \right) \, dy \\
    &= - (\sigma_c - 1) \int_D \nabla_y G(x,y) \cdot \nabla_y u(y) \, dy. 
\end{align*}
Then, we have 
\begin{equation*}
    \pa_{x_i} w(x) = - (\sigma_c - 1) \int_D \pa_{x_i} \nabla_y G(x,y) \cdot \nabla_y u(y) \, dy. 
\end{equation*}
Now, by Lemma \ref{esti for green} and the Cauchy--Schwarz inequality, we may write (as usual, up to redefining $C^*$) 
\begin{align}\label{esti0}
    |\nabla_x w| &\leq \sqrt{N} |\sigma_c - 1| \int_D |\nabla_x \nabla_y G(x,y)| |\nabla_y u(y)| \, dy \notag \\
    &\leq |\sigma_c - 1| C^* M^{N+1} |D|^{1/2}\left( \int_\Omega |\nabla_y u|^{2} \, dy \right)^{1/2}.
\end{align}
Consider the weak form of \eqref{eq:2phase problem}. For any $\varphi \in H^1_0(\Omega)$, 
\begin{equation*}
    \int_{\Omega} \sigma \nabla_y u \cdot \nabla_y \varphi \, dy = \int_{\Omega} \varphi \, dy. 
\end{equation*}
Taking $\varphi = u$, then 
\begin{equation}\label{esti1}
\min\{\sigma_c,1\} \int_{\Omega} |\nabla_y u|^2 \, dy \leq \int_{\Omega} \sigma |\nabla_y u|^2 \, dy = \int_{\Omega} u \, dy \leq |\Omega|^{1/2} \norm{u}_{L^2(\Omega)}.    
\end{equation}
Let now $\la_1(\Om)$ denote the first eigenvalue of the Laplace operator with Dirichlet zero boundary condition, that is 
\begin{equation*}
    \la_1(\Om)=\inf_{f\in H^1_0(\Om),\; f\not\equiv 0} \frac{\norm{\gr f}_{L^2(\Om)}}{\norm{f}_{L^2(\Om)}}.
\end{equation*}
Since $u \in H^1_0(\Omega)$ and $u\not\equiv0$, we have $\lambda_1(\Omega) \leq \norm{\nabla_y u}^2_{L^2(\Omega)} \mathlarger{\mathlarger{\mathlarger{\mathlarger{/}}}} \norm{u}^2_{L^2(\Omega)}$. Thus, we obtain 
\begin{equation}\label{esti2}
    \norm{u}_{L^2(\Omega)} \leq \dfrac{\norm{\nabla_y u}_{L^2(\Omega)}}{\lambda^{1/2}_1(\Omega)}. 
\end{equation}
Combining \eqref{esti1} with \eqref{esti2}, 
\begin{equation}\label{esti3}
    \norm{\nabla_y u}_{L^2(\Omega)} \leq \dfrac{|\Omega|^{1/2}}{\lambda^{1/2}_1(\Omega)\min\{\sigma_c,1\}}. 
\end{equation}
By \eqref{esti0} and \eqref{esti3}, we have 
\begin{equation*}
    |\nabla_x w| \leq |\sigma_c - 1| C^* M^{N+1} |D|^{1/2} \dfrac{ |\Omega|^{1/2}}{\lambda^{1/2}_1(\Omega)\min\{\sigma_c,1\}}.
\end{equation*}
By the Faber--Krahn inequality, there exists a ball $B^\star$ such that 
\begin{equation*}
    \lambda_1(B^\star) \leq \lambda_1(\Omega), \quad |B^\star| = |\Omega|. 
\end{equation*}
Therefore, we obtain 
\begin{equation}\label{esti:grad w}
    |\nabla_x w| \leq |\sigma_c - 1| C^* M^{N+1} |D|^{1/2} \dfrac{ |\Omega|^{1/2}}{\lambda^{1/2}_1(B^\star)\min\{\sigma_c,1\}}.
\end{equation}
By the classical Schauder estimates for the Dirichlet problem near the boundary (see, for instance \cite[Theorem 8.33]{GT} and the subsequent remarks), $\nabla_x w$ is continuous up to the boundary $\pa \Omega$. If we let $x$ tend to $\pa \Omega$, then we realize that \eqref{esti:grad w} also holds true for $x\in\pa\Om$. 
Let us recall that the solution $u$ of \eqref{eq:2phase problem} satisfies the overdetermined condition $\pa_n u = c$ on $\pa \Omega$. Therefore, for any $x\in\pa\Om$, we obtain 
\begin{equation*}
    |\pa_n v - c| \leq |\nabla_x w| \leq |\sigma_c - 1| |D|^{1/2}\dfrac{C^* M^{N+1}|\Omega|^{1/2}}{\lambda^{1/2}_1(B^\star)\min\{\sigma_c,1\}}.
\end{equation*}
By applying Corollary \ref{cor:Improved-Serrin-stability-deviation in uniformnorm}, we get the following estimate: 
\begin{equation*}
\rho_e-\rho_i\le C_3 |D|^{\frac{\tau_N}{2}},   
\end{equation*}
where $\tau_N$ is defined as in Theorem \ref{thm:Improved-Serrin-stability}.
The constant $C_3$ depends on $N$, $d_\Om$, $\sg_c$, $M$, the $C^{1,\al}$ modulus of $\pa\Om$, and $\theta$ (only in the case $N=3$). 
\end{proof}

\begin{remark}
It is clear that the proof of Theorem \ref{thm II} can also be used to obtain a stability estimate in the spirit of Theorem \ref{thm I}. However, notice that such a proof would lead to a (weaker) version of Theorem \ref{thm I} in which the constant $C_2$ also depends on the distance between $D$ and $\pa \Om$, and $|D|$. 
\end{remark}

\begin{remark}
Whenever an apriori bound for $\nr \na u \nr_{L^{\infty}(D)}$ is available, the stability exponent of Theorem \ref{thm II} can be improved (that is $\tau_N /2$ can be replaced by $\tau_N$), at the cost of allowing the constant $C_3$ to depend also on the above mentioned bound for $\nr \na u \nr_{L^{\infty}(D)}$. This can be obtained by replacing in the proof of Theorem \ref{thm II}, \eqref{esti0} with
$$
| \na_x w| \le | \sigma_c -1| C^* M^{N+1} |D| \, \nr \na_y u \nr_{L^\infty (D)} .
$$

\end{remark}

\section{Non-existence for the inner problem when $\si_c\simeq 1$ or $|D|$ is small }\label{sec:nonexistence}
In this section we show how one can employ the results of Theorems \ref{thm I} and \ref{thm II} to prove non-existence for the inner problem corresponding to \eqref{odp} when $\si_c\simeq 1$ or $|D|$ is small.  

\begin{proof}[Proof of Corollary \ref{thm III}]
Let $\Om$ be a bounded domain of $\rn$ different from a ball and set $c:=-|\Om|/|\pa\Om|$. 
Since, by hypothesis, $\Om$ is not a ball, we have
\begin{equation*}
\rho_e-\rho_i>0.
\end{equation*}
Now, let $C_4$ and $\tau_N$ be the same constants as in the statement of Corollary \ref{thm III} 
 and suppose by contradiction that there is an open set $D\subset\ol D\subset\Om$ such that the overdetermined problem \eqref{odp} admits a solution $u$ for some $\sg_c$ satisfying 
\begin{equation}\label{by contradiction}
|\sg_c-1| < 
    C_4\left( {\rho_e-\rho_i} \right)^{\frac{1}{\tau_N}},    
\end{equation} 
(notice that there exist infinitely many such values of $\sg_c$ because $\rho_e-\rho_i>0$ by construction).
Finally, Theorem \ref{thm I} yields 
\begin{equation*}
\rho_e-\rho_i\le C_2|\sg_c-1|^{\tau_N}<C_2\  C_4^{\tau_N} (\rho_e-\rho_i)=\rho_e-\rho_i,    
\end{equation*}
which is a contradiction.
\end{proof}

\begin{remark}
By applying the result of Remark \ref{extension of thm I}, we can extend Corollary \ref{thm III} to the case where the overdetermined condition in \eqref{odp} is replaced by \eqref{extended odc} for some $\eta\in L^\infty(\pa\Om)$ with vanishing mean over $\pa\Om$. In this case, given a bounded domain $\Om$ of class $C^{1,\al}$ that is not a ball, the overdetermined problem given by \eqref{eq:2phase problem} and \eqref{extended odc} does not admit a solution of the form $(D,\Om)$ if \begin{equation*}
    |\sg_c-1| < 
    \frac{1}{C_7}\left\{\left( \frac{\rho_e-\rho_i}{C_6} \right)^{\frac{1}{\tau_N}} -\norm{\eta}_{L^\infty(\pa\Om)}  \right\}, 
\end{equation*}
where $C_6$ and $\tau_N$ are as in Corollary \ref{cor:Improved-Serrin-stability-deviation in uniformnorm}, while $C_7$ is the constant in \eqref{v close to neumann}.
Notice that the set of values $\sg_c$ satisfying the inequality above is not empty if the norm $\norm{\eta}_{L^\infty(\pa\Om)}$ is small enough. 
\end{remark}

\begin{proof}[Proof of Corollary \ref{thm IV}]
It follows from  Theorem \ref{thm II} by arguing by contradiction. The proof will be omitted because it is completely analogous to that of Corollary \ref{thm III}. 
\end{proof}

\begin{figure}[h]
\centering
\includegraphics[width=0.9\linewidth]{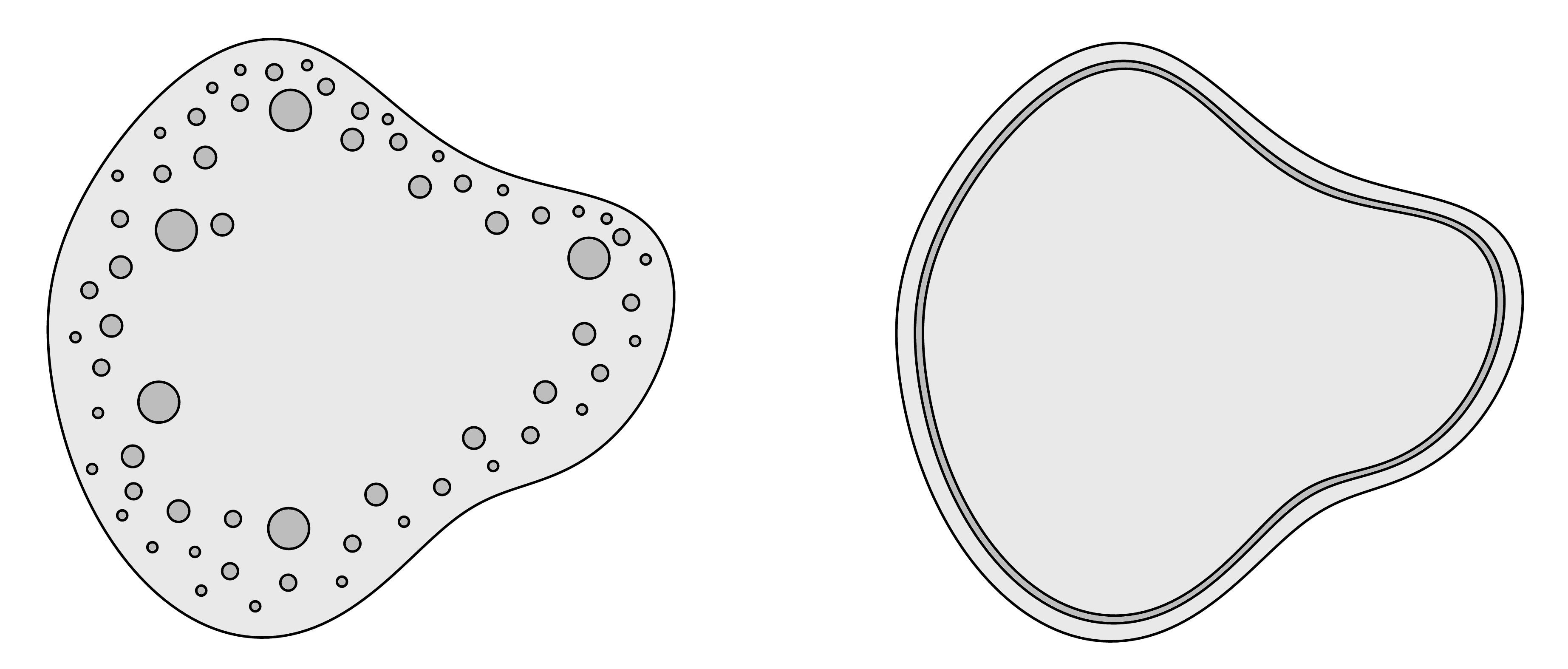}
\caption{Two examples of wild solutions. Left: due to the formation of microstructures. Right: due to boundary layer effect.} 
\label{wildsolutions}
\end{figure}

\begin{remark}
We remark that the constant $C_5$ of Corollary \ref{thm IV} depends on the distance between $D$ and $\pa\Om$. Indeed, given $\Om$, $\sg_c>0$ and $M>1$, Corollary \ref{thm IV} tells us that there does not exist a solution of \eqref{odp} of the form $(D,\Om)$, where $D$ is an open set belonging to the class
\begin{equation*}
\cD_M:=\left\{ D\subset
\Om\;:\; \dist(D,\pa\Om)\ge \frac 1 M
\right\}   
\end{equation*}
and the volume $|D|$ is small enough (namely, smaller than $C_5\left( {\rho_e-\rho_i} \right)^{\frac{2}{\tau_N}}$).

Indeed, Corollary \ref{thm IV} does not preclude the existence of a family of ``wild solutions" $\left\{(D_k,\Om)\right\}_{k\ge 1}$ of \eqref{odp} with $D_k\in\cD_{M_k}$ such that 
\begin{equation*}
    \lim_{k\to \infty}M_k=\infty.
\end{equation*}
Geometrically speaking, this suggests the possibility of ``wild solutions" $(D_k,\Om)$ where the inclusion $D_k$ becomes closer and closer to the boundary $\pa\Om$ as $k\to \infty$.
We conjecture that this could happen in many ways. For example, when $D$ takes the form of a thin layer increasingly close to $\pa\Om$ or when one allows the formation of increasingly many connected components that give rise to a microstructure. Indeed, both such configurations seem likely to affect the global behavior of the solution of \eqref{eq:2phase problem} near the boundary (see Figure \ref{wildsolutions}). Such behaviors are linked to the so-called homogenization phenomena (see \cite{BCF1980, Fr1980, MT1997A, MT1997B, Y2-2019, All2019} and the references therein). 
\end{remark}

\section*{Acknowledgements}
The first author is partially supported by JSPS Grant-in-Aid for Research Activity Start-up Grant Number JP20K22298.
The second author is supported by the Australian Laureate Fellowship FL190100081 ``Minimal surfaces, free boundaries and partial differential equations'' and is member of AustMS and INdAM/GNAMPA.
The third author is partially supported by JSPS Grant-in-Aid for Early-Career Scientists Grant Number JP21K13822. 
%
%

\begin{small}

\end{small}

\bigskip

\noindent
\textsc{
Mathematical Institute, Tohoku University, Aoba-ku, 
Sendai 980-8578, Japan}\\
\noindent
{\em Electronic mail address:}
cavallina.lorenzo.e6@tohoku.ac.jp

\bigskip

\noindent
\textsc{
Department of Mathematics and Statistics, The University of Western Australia, 35 Stirling Highway, Crawley, Perth, WA 6009, Australia} \\
\noindent
{\em Electronic mail address:}
giorgio.poggesi@uwa.edu.au

\bigskip

\noindent
\textsc{ 
Kyoto University Institute for Advanced Study, Sakyo-ku, 
Kyoto 606-8501, Japan } \\
\noindent
{\em Electronic mail address:}
yachimura.toshiaki.8n@kyoto-u.ac.jp

\end{document}